\documentclass[11pt,reqno]{amsart}
\usepackage{amsthm}
\usepackage{amssymb}
\usepackage{latexsym}
\usepackage{multicol}
\usepackage{verbatim}
\usepackage{amscd}
\usepackage{hyperref}
\usepackage{amsmath, amscd}
\usepackage{color}
\usepackage[all]{xy}
\advance\textwidth by 1.2in \advance\oddsidemargin by -.6in \advance\evensidemargin by -.6in
\newtheorem*{cor}{Corollary}%[section]
\newtheorem*{lem}{Lemma}
\newtheorem*{prop}{Proposition}

\theoremstyle{definition}

\theoremstyle{definition}
\newtheorem*{thm}{Theorem}

\newtheorem*{rem}{Remark}

\newenvironment{pf}{\proof}{\endproof}
\newcounter{cnt}
\newenvironment{enumerit}{\begin{list}{{\hfill\rm(\roman{cnt})\hfill}}{%
\settowidth{\labelwidth}{{\rm(iv)}}\leftmargin=\labelwidth%
\advance\leftmargin by \labelsep\rightmargin=0pt\usecounter{cnt}}}{\end{list}} \makeatletter
\def\mydggeometry{\makeatletter\dg@YGRID=1\dg@XGRID=20\unitlength=0.003pt\makeatother}
\makeatother \theoremstyle{remark}

\numberwithin{equation}{section}

\let\bwdg\bigwedge
\def\bigwedge{{\textstyle\bwdg}}

\begin{document}

\newcommand{\thmref}[1]{Theorem~\ref{#1}}
\newcommand{\secref}[1]{Section~\ref{#1}}
\newcommand{\lemref}[1]{Lemma~\ref{#1}}
\newcommand{\propref}[1]{Proposition~\ref{#1}}
\newcommand{\corref}[1]{Corollary~\ref{#1}}
\newcommand{\remref}[1]{Remark~\ref{#1}}
\newcommand{\defref}[1]{Definition~\ref{#1}}
\newcommand{\er}[1]{(\ref{#1})}
\newcommand{\id}{\operatorname{id}}
\newcommand{\ord}{\operatorname{\emph{ord}}}
\newcommand{\sgn}{\operatorname{sgn}}
\newcommand{\wt}{\operatorname{wt}}
\newcommand{\tensor}{\otimes}
\newcommand{\from}{\leftarrow}
\newcommand{\nc}{\newcommand}
\newcommand{\rnc}{\renewcommand}
\newcommand{\dist}{\operatorname{dist}}
\newcommand{\qbinom}[2]{\genfrac[]{0pt}0{#1}{#2}}
\nc{\cal}{\mathcal} \nc{\goth}{\mathfrak} \rnc{\bold}{\mathbf}
\renewcommand{\frak}{\mathfrak}
\newcommand{\supp}{\operatorname{supp}}
\newcommand{\Irr}{\operatorname{Irr}}
\renewcommand{\Bbb}{\mathbb}
\nc\bomega{{\mbox{\boldmath $\omega$}}} \nc\bpsi{{\mbox{\boldmath $\Psi$}}}
 \nc\balpha{{\mbox{\boldmath $\alpha$}}}
 \nc\bpi{{\mbox{\boldmath $\pi$}}}
\nc\bsigma{{\mbox{\boldmath $\sigma$}}} \nc\bcN{{\mbox{\boldmath $\cal{N}$}}} \nc\bcm{{\mbox{\boldmath $\cal{M}$}}} \nc\bLambda{{\mbox{\boldmath
$\Lambda$}}}
\newcommand{\bdd}{\operatorname{bdd}}
\newcommand{\conv}{\operatorname{conv}}
\newcommand{\lie}[1]{\mathfrak{#1}}
\makeatletter
\def\section{\def\@secnumfont{\mdseries}\@startsection{section}{1}%
  \z@{.7\linespacing\@plus\linespacing}{.5\linespacing}%
  {\normalfont\scshape\centering}}
\def\subsection{\def\@secnumfont{\bfseries}\@startsection{subsection}{2}%
  {\parindent}{.5\linespacing\@plus.7\linespacing}{-.5em}%
  {\normalfont\bfseries}}
\makeatother
\def\subl#1{\subsection{}\label{#1}}
 \nc{\Hom}{\operatorname{Hom}}
  \nc{\mode}{\operatorname{mod}}
\nc{\End}{\operatorname{End}} \nc{\wh}[1]{\widehat{#1}} \nc{\Ext}{\operatorname{Ext}} \nc{\ev}{\operatorname{ev}}
\nc{\Ob}{\operatorname{Ob}} \nc{\soc}{\operatorname{soc}} \nc{\rad}{\operatorname{rad}} \nc{\head}{\operatorname{head}}
\def\Im{\operatorname{Im}}
\def\gr{\operatorname{gr}}
\def\ch{\operatorname{ch}}
\def\mult{\operatorname{mult}}
\def\Max{\operatorname{Max}}
\def\ann{\operatorname{Ann}}
\def\sym{\operatorname{sym}}
\def\Res{\operatorname{\br^\lambda_A}}
\def\und{\underline}
\def\Lietg{$A_k(\lie{g})(\bsigma,r)$}

 \nc{\Cal}{\cal} \nc{\Xp}[1]{X^+(#1)} \nc{\Xm}[1]{X^-(#1)}
\nc{\on}{\operatorname} \nc{\Z}{{\bold Z}} \nc{\J}{{\cal J}} \nc{\C}{{\bold C}} \nc{\Q}{{\bold Q}}
\renewcommand{\P}{{\cal P}}
\nc{\N}{{\Bbb N}} \nc\boa{\bold a} \nc\bob{\bold b} \nc\boc{\bold c} \nc\bod{\bold d} \nc\boe{\bold e} \nc\bof{\bold f} \nc\bog{\bold g}
\nc\boh{\bold h} \nc\boi{\bold i} \nc\boj{\bold j} \nc\bok{\bold k} \nc\bol{\bold l} \nc\bom{\bold m} \nc\bon{\bold n} \nc\boo{\bold o}
\nc\bop{\bold p} \nc\boq{\bold q} \nc\bor{\bold r} \nc\bos{\bold s} \nc\boT{\bold t} \nc\boF{\bold F} \nc\bou{\bold u} \nc\bov{\bold v}
\nc\bow{\bold w} \nc\boz{\bold z} \nc\boy{\bold y} \nc\ba{\bold A} \nc\bb{\bold B} \nc\bc{\bold C} \nc\bd{\bold D} \nc\be{\bold E} \nc\bg{\bold
G} \nc\bh{\bold H} \nc\bi{\bold I} \nc\bj{\bold J} \nc\bk{\bold K} \nc\bl{\bold L} \nc\bm{\bold M} \nc\bn{\bold N} \nc\bo{\bold O} \nc\bp{\bold
P} \nc\bq{\bold Q} \nc\br{\bold R} \nc\bs{\bold S} \nc\bt{\bold T} \nc\bu{\bold U} \nc\bv{\bold V} \nc\bw{\bold W} \nc\bz{\bold Z} \nc\bx{\bold
x} \nc\KR{\bold{KR}} \nc\rk{\bold{rk}} \nc\het{\text{ht }}
\nc\loc{\rm{loc }}

\nc\toa{\tilde a} \nc\tob{\tilde b} \nc\toc{\tilde c} \nc\tod{\tilde d} \nc\toe{\tilde e} \nc\tof{\tilde f} \nc\tog{\tilde g} \nc\toh{\tilde h}
\nc\toi{\tilde i} \nc\toj{\tilde j} \nc\tok{\tilde k} \nc\tol{\tilde l} \nc\tom{\tilde m} \nc\ton{\tilde n} \nc\too{\tilde o} \nc\toq{\tilde q}
\nc\tor{\tilde r} \nc\tos{\tilde s} \nc\toT{\tilde t} \nc\tou{\tilde u} \nc\tov{\tilde v} \nc\toW{\tilde w} \nc\toz{\tilde z}

\title{Macdonald Polynomials and BGG reciprocity for current algebras}
\author{Matthew Bennett, Arkady Berenstein, Vyjayanthi Chari,\\ \\ Anton Khoroshkin and Sergey Loktev }
\email{mbenn002@gmail.com\iffalse  Department of Mathematics, University of California, Riverside, CA 92521,
 U.S.A.\fi}
 \email{arkadiy@math.uoregon.edu}
\email{chari@math.ucr.edu\iffalse Department of Mathematics, University of California, Riverside, CA 92521,
 U.S.A.\fi}
 \email{khoroshkin@gmail.com}
  \email{s.loktev@gmail.com \iffalse Higher School of Economics, 7 Vavilova Str. Moscow, Russia.\fi }
\thanks{V.C. was partially supported by DMS-0901253}\thanks{S.L. was partially supported by RFBR-CNRS-11-01-93105 and
RFBR-12-01-00944.}
\iffalse\author{Arkady Berenstein, }
\address{}
\email{}
\author{Vyjayanthi Chari}
 \address{Department of Mathematics, University of California, Riverside, CA 92521,
 ]U.S.A.}
 \email{chari@math.ucr.edu}
 \author{Anton Khoroshkin}
 \address{}
 \email{}
 \author{Sergey Loktev}
 \address{Higher School of Economics, 7 Vavilova Str. Moscow}
 \email{s.loktev@gmail.com}
 \thanks{S.L. was partially supported by RFBR-CNRS-11-01-93105 and
RFBR-12-01-00944.}
 \email{ vyjayanthi.chari@ucr.edu}\thanks{V.C. was partially supported by DMS-0901253}\fi
\maketitle

\begin{abstract} We study the category $\cal I_{\gr}$  of graded representations with finite--dimensional graded pieces for the current algebra $\lie g\otimes\bc[t]$ where $\lie g$ is a simple Lie algebra. This category has   many similarities with the category $\cal O$ of modules for $\lie g$ and  in this paper and we prove an   analogue of the famous BGG duality in the case of $\lie{sl}_{n+1}$. 
\end{abstract}

\section*{Introduction}

 The current algebra associated to a simple Lie algebra is just the Lie algebra of polynomial maps from $\bc\to \lie g$ and can be identified with the space $\lie g\otimes \bc[t]$ with the obvious commutator. Another way of thinking of this is as a maximal parabolic subalgebra in the corresponding untwisted affine Kac--Moody algebra. The Lie algebra and its universal enveloping algebra inherit a  grading coming from the natural grading on $\bc[t]$. We are  interested in the category $\cal I$ of $\bz$--graded modules for  $\lie g[t]$ with the restriction that the graded pieces are finite--dimensional. Originally, the study of this category was largely motivated by its relationship to the representation theory of affine and quantum affine algebras associated to a simple Lie algebra $\lie g$. However,  it is also now of independent interest since it yields connections with problems arising in mathematical physics, for instance  the $X=M$ conjectures, see \cite{AK}, \cite{deFK}, \cite{Naoi}.

 The category $\cal I$ is a non--semisimple category and has many similarities with other well--known categories of representations in Lie theory. However, there are many essential differences in the theory as we shall see below, which makes it quite remarkable that one can formulate  (see \cite{BCM} ) of the famous Bernstein--Gelfand--Gelfand (BGG)-reciprocity result for the category $\cal O$. In \cite{BCM} the result was proved for $\lie{sl}_2$ by different methods. In the current paper, we use the combinatorics of Macdonald polynomials to  extend the result to $\lie{sl}_{n+1}$.

The main ingredients in the original theorem of Bernstein--Gelfand--Gelfand were the irreducible modules $V(\lambda)$ for a simple Lie algebra, the Verma module $M(\lambda)$ and the projective cover $P(\lambda)$ of $V(\lambda)$ where $\lambda$ is a linear functional on a Cartan subalgebra of $\lie g$. The Verma modules have a nice freeness property and  it is relatively easy to prove that the projective module has a filtration by Verma modules. Further, the Verma modules have Jordan--Holder series and the  theorem states that the filtration multiplicity of the Verma module $M(\mu)$ in the projective $P(\lambda)$ is equal to the  multiplicity of $V(\lambda)$ in $M(\mu)$.

The irreducible objects in $\cal I$ are indexed by two parameters, $(\lambda,r)$ where $\lambda$ varies over the index set of irreducible finite--dimensional representations of $\lie g$ and $r$ varies over the integers. The category $\cal I$ also contains the projective covers $P(\lambda,r)$ of the simple object $V(\lambda,r)$. The appropriate analog of the Verma module is the global Weyl module  $W(\lambda,r)$ defined originally in \cite{CPweyl} via generators and relations. It is in fact the maximal quotient of $P(\lambda,r)$ with respect to the property that the eigenvalues of $\lie h$ lie in a certain finite set. At this point two points of similarity fail: the global Weyl modules do have a nice freeness property, but it is for a much smaller algebra than in the case of $\lie g$. Thus, we have to work harder to prove that the projective modules have a filtration by global Weyl modules. We use an idea from algebraic groups (see \cite{Donkin}) and define a canonical filtration on a object of $\cal I$ and show that the successive quotients of the filtration are isomorphic to a quotient of a direct sum of global Weyl modules.

The second difficulty we encounter is that the global Weyl modules are not of finite length. To circumvent this, we recall that they have a unique maximal finite--dimensional quotient called the local Weyl modules (see \cite{CFK}, \cite{CPweyl})  and this allows us to formulate the desired result. Namely, the projective module $P(\lambda,r)$ has a filtration by global Weyl modules and the multiplicity of $W(\mu,s)$ in $P(\lambda,r)$ is the multiplicity of $V(\lambda,s)$ in the local Weyl module  $W_{\loc}(\mu, r)$. This result was proved in \cite{BCM} in the case of $\lie{sl}_2$ and conjectured to be true in general.

In this paper, we are able to prove that that the conjecture is true iff the canonical filtration of $P(\lambda,r)$ is  actually a filtration by global Weyl modules. We then establish that this is true for  $\lie{sl}_{n+1}$.  To explain this restriction and the connection with Macdonald polynomials, we need some further comments on local Weyl modules. It was proved in \cite{CL} that for $\lie{sl}_{n+1}$ the local Weyl module is isomorphic to a Demazure module in a level one representation of the affine Kac--Moody algebra. It was proved in \cite{S} that the character of such a Demazure module is given by specialization of a  Macdonald polynomial at $t=0$. In Section 4, we use several properties of Macdonald polynomials to establish certain combintorial  identities. In Section 5, we prove that these identities have a representation theoretic interpretation, namely they give a relation between the Hilbert series of $P(\lambda,r)$ and a sum of Hilbert series of  global Weyl modules (with multiplicity). This is enough to establish the reciprocity result. In the general simply laced case, it is still true that the local Weyl modules are Demazure modules and their characters are given in \cite{I} via non--symmetric Macdonald polynomials. In the non--simply laced case, it was proved in \cite{Naoi} that local Weyl modules  have a filtration by Demazure modules  and the characters are known. The missing piece in the case when $\lie g$ is not of type $\lie{sl}_{n+1}$ is thus the  combinatorial problem studied in Section 4.  It is necessary is to establish the correct version of Lemma \ref{crucial} and we will return to this elsewhere.

{\em Acknowledgements: We thank Boris Feigin for stimulating discussions. It is a pleasure for the second, third and fifth authors to thank the organizers of  the  trimester  \lq\lq On the interactions of Representation theory with Geometry and Combinatorics,\rq\rq at the Hausdorff Institute, Bonn, 2011, when much of this work was done.}
\section{Preliminaries}

\subsection{} Throughout this paper we denote by   $\bc$  the field of complex numbers and $\bz$ (resp. $\bz_+$) the set of integers (resp.  nonnegative   integers). For a Lie algebra $\lie a$ denote by
 $\bu(\lie a)$ the universal enveloping algebra of $\lie a$. If $t$ is  an indeterminate, let $\lie a[t]=\lie a\otimes \bc[t]$ be  the Lie algebra with commutator given by, $$[a\otimes f, b\otimes g]=[a,b]\otimes fg,\ \ a, b\in\lie a, \ \ f,g,\in\bc[t].$$  We identify $\lie a$ with the Lie subalgebra $\lie a\otimes 1$ of $\lie a[t]$. The Lie algebra $\lie a[t]$ has a natural $\bz_+$--grading given by the powers of $t$ and this also induces a $\bz_+$--grading on $\bu(\lie a[t])$, and  $\bu(\lie a[t])[0]=\bu(\lie a)$. The  graded pieces of $\bu(\lie a[t])$ are  $\lie a$--modules under left and right multiplication by elements of $\lie a$ and hence also under the adjoint action of $\lie a$. \iffalse In particular, if $\dim\lie a<\infty$, then $\bu(\lie a[t])[r]$ is a free module for $\lie a$ (via left or right multiplication) of finite rank.\fi

\subsection{}
From now on, $\lie g$  denotes a finite--dimensional complex simple Lie algebra of rank $n$ and $\lie h$   a fixed Cartan subalgebra of $\lie g$. Let   $I=\{1,\cdots ,n\}$  and fix a set  $\{\alpha_i: i\in I\}$ of simple roots of $\lie g$ with respect to $\lie h$ and a set $\{\omega_i: i\in I\}$   of fundamental weights.  Let $Q$ (resp. $Q^+$) be  the integer span (resp. the nonnegative integer span) of $\{\alpha_i: i\in I\}$ and similarly define  $P$ (resp. $P^+$) to be the $\bz$ (resp. $\bz_+$) span of  $\{\omega_i: i\in I\}$. Let $\{x_i^\pm, h_i: i\in I\}$ be a set of Chevalley generators of $\lie g$ and let $\lie n^\pm$ be the Lie subalgebra of $\lie g$ generated by the elements $x_i^\pm$, $i\in I$. We have, $$\lie g\ =\ \lie n^-\oplus\lie h\oplus\lie n^+,\ \ \qquad \bu(\lie g)=\bu(\lie n^-)\otimes\bu(\lie h)\otimes\bu(\lie n^+). $$ Let $W$ be the Weyl group of $\lie g$ and let $w_0\in W$ be the longest element of $W$.  Given $\lambda,\mu\in\lie h^*$, we say that  $\lambda\le\mu$ iff $ \lambda-\mu\in Q^+$.

\subsection{}\label{locfin}
For any $\lie g$-module $M$ and $\mu\in\lie h^*$, set
\[M_\mu=\{m\in M\ :\ hm=\mu(h)m,\quad h\in\lie h\},\ \ \wt(M)=\left\{\mu\in\lie h^\ast\ :\ M_\mu\ne 0\right\}.\] We say  $M$ is a \textit{weight module}
for $\lie g$ if \[M=\bigoplus_{\mu\in\lie h^\ast} M_\mu.\]
Any finite--dimensional $\lie g$--module is  a weight module.
   It is well-known that the set of isomorphism classes of irreducible finite-dimensional $\lie g$-modules is in bijective correspondence with $P^+$. For  $\lambda\in P^+$ we denote by $V(\lambda)$ a representative of the corresponding isomorphism class.  Then $V(\lambda)$  is generated as a $\lie g$--module by  a vector $v_\lambda$ with defining relations
  \[ \lie n^+v_\lambda=0,\qquad hv_\lambda=\lambda(h)v_\lambda,\qquad (x_{{i}}^-)^{\lambda(h_{i})+1}v_{\lambda}=0,\ \ \  h\in\lie h,\ \ i\in I.\]
 and recall that $\wt V(\lambda)\subset\lambda-Q^+$. The module $V(0)$ is the trivial module for $\lie g$ and we shall write it as $\bc$. Let $\bz[P]$ be  the integral group ring $\bz[P]$ spanned by elements $e(\mu)$, $\mu\in P$ and given a finite--dimensional $\lie g$--module,
let
\[ {\rm ch}_{\lie g} M=\sum_{\mu\in P}\dim_\bc M_\mu e(\mu).\]
The  set $\{\ch_{\lie g} V(\mu): \mu\in P^+\}$ is a  linearly independent subset of $\bz[P]$.

We say that $M$ is a \textit{locally finite-dimensional} $\lie g$--module if it is a direct sum of finite--dimensional $\lie g$--modules,  in which case $M$ is necessarily a weight module.
Using Weyl's theorem one knows that  a locally finite-dimensional $\lie g$-module $M$ is isomorphic to a direct sum of modules of the form $V(\lambda)$, $\lambda\in P^+$ and  hence  $\wt M\subset P$.\iffalse Set \begin{equation}\label{ndecomp} M^{\lie n^+}=\{m\in M: \lie n^+ m=0,\}\ \qquad M^{\lie n^+}_\lambda=M^{\lie n^+}\cap M_\lambda\cong\Hom_{\lie g}(V(\lambda), M). \end{equation}
\fi

\subsection{}

Let $\cal I$ be the  category whose objects are graded $\lie g[t]$-modules $V$ with finite-dimensional graded components and where the morphisms are  maps  of graded $\lie g[t]$-modules.  Thus an object $V$ of $\cal I$, is a $\bz$--graded vector space  $V =\oplus_{s\in\bz}V[s]$, $\dim V[s]<\infty$ which admits a left action of $\lie g[t]$ satisfying
$$(\lie g\otimes t^r)V[s]\subset V[s+r],\qquad s,r\in\bz.$$ For all $r\in\bz$, the subspace $V[r]$ is a  finite--dimensional $\lie g$--module.
A morphism between  two objects $V$, $W$ of $\cal I$ is a degree zero  map of graded $\lie g[t]$--modules.  Clearly $\cal I$ is closed under taking submodules, quotients and finite direct sums. For any $r\in\bz$ we let $\tau_r$ be the grade shifting operator.
\iffalse
 If $V\in\Ob\cal I$ and $\mu\in P^+$, then $$V_\mu^{\lie n^+}= \bigoplus_{r\in\bz} V[r]_\mu^{\lie n^+},\ \qquad V[r]_\mu^{\lie n^+}= V_\mu^{\lie n^+}\cap V[r].$$\fi
  The \textit{graded character}  (resp. Hilbert series) of $V\in\Ob\mathcal I$ is the element of the space of  power series $\bz[P][[q,q^{-1}]]$, given by
\[{\rm ch}_{\rm{gr}}V=\sum_{r\in\bz} {\rm ch}_\lie{g}(V[r])q^{r},\qquad\ \ \mathbb H(V)=\sum_{r\in\bz}\dim V[r]q^r.\]
Given $V\in\Ob\cal I$, the restricted dual is $$V^*=\bigoplus_{r\in\bz} V[r]^*,\qquad  V^*[r]=V[-r]^*.$$  Then $V^*\in\Ob\cal I$ with the usual action:$$(xt^s)v^*(w)=-v^*(xt^s w),$$ and $(V^*)^*\cong V$ as objects of $\cal I$. Note that if $V\in\Ob{\cal I}$, then \[{\rm ch}_{\rm{gr}}V^*:=\sum_{r\in\bz} {\rm ch}_\lie{g}(V[r]^*)u^{-r}.\]

\section{The main result}

\subsection{} Let $\ev_0: \lie g[t]\to\lie g$ be the homomorphism of Lie algebras which maps $x\otimes f\mapsto f(0)x$. The kernel of this map is a graded ideal in $\lie g[t]$ and hence any $\lie g$--module $V$ can be regarded in an obvious way
as a graded $\lie g[t]$--module denoted $\ev_0V$.  Clearly $\ev_0V$ is an object of $\cal I$ if $\dim V < \infty$.
 The pull back of $V(\lambda)$ is denoted $V(\lambda,0)$ and we set $\tau_rV(\lambda,0)=V(\lambda,r)$ and we let $v_{\lambda,r}\in V(\lambda,r)$ be the element corresponding to $v_\lambda$. The following is elementary and a proof can be found in \cite{CG}.
 \begin{lem}\label{irred}
Any irreducible  object  in $\cal I$  is isomorphic to $V(\mu,r)$ for a unique element
  $(\mu,r)\in P^+\times \bz$ and $V(\mu,r)^*\cong V(-w_0\mu,-r)$. Further $V\in\Ob{\cal I}$ is semisimple iff $$V\cong\oplus_{(\lambda,r)\in P^+\times\bz}V(\lambda,r)^{m(\lambda,r)},
   \ \ \ m(\lambda,r)\in\bz_+.$$

  \iffalse Moreover, $\ch_{\gr} V(\lambda,r)=u^r\ch_{\lie g}V(\lambda)$ and
 $\{\ch_{\gr} V(\lambda,r):(\lambda,r)\in P^+\times\bz\}$,   is a linearly independent subset of $\bz[P][[u,u^{-1}]]$.\fi\hfill\qedsymbol\end{lem}
  Suppose that $\dim V<\infty$ and  $r$ is minimal such that $V[r]\ne 0$. Then we have a short exact sequence of $\lie g[t]$--modules $$0\to\bigoplus_{s>r}V[s]\to V\to \ev_0 V[r]\to 0.$$ A simple induction on $\dim V$ now proves that for all $(\lambda,s)\in P^+\times\bz$, we have,
 \begin{equation}\label{jhmult} [V:V(\lambda,s)]=\dim\Hom_{\lie g}(V(\lambda), V[s]) =\dim\Hom_{\lie g}(V[s],V(\lambda)),\end{equation} where $[V:V(\lambda,s)]$ is the multiplicity of $V(\lambda,s)$ in a Jordan--Holder series of $V$.

 \subsection{}\label{weylloc}  For $\lambda\in P^+$ and $r\in\bz$, the local Weyl module, $W_{\loc}(\lambda,r)$,  is the  $\lie g[t]$--module generated by an element $w_{\lambda,r}$ with relations:\begin{gather*} \lie n^+[t]w_{\lambda,r}=0,\qquad (x_{i}^-)^{\lambda(h_{i})+1}w_{\lambda,r}=0,\\
 (h\otimes t^s)w_{\lambda,r}= \delta_{s,0}\lambda(h)w_{\lambda,r},
 \end{gather*} where $i\in I$, $h\in\lie h$ and $s\in\bz_+$. The next proposition summarizes the results on the local Weyl module which are needed to state our main result. A proof of this proposition can be found in \cite{CPweyl}.
   \begin{prop} Let $(\lambda,r)\in P^+\times\bz$.
  Let $(\lambda,r)\in P^+\times\bz$.  The $\lie g[t]$--module $W_{\loc}(\lambda,r)$ is indecomposable and finite--dimensional. Moreover,  $\dim W_{\loc}(\lambda,r)_\lambda =\dim W_{\loc}(\lambda,r)[r]_\lambda=1$,
 and  $V(\lambda,r)$ is the unique irreducible quotient of $W_{\loc}(\lambda,r)$.
 \hfill\qedsymbol
\end{prop}

\subsection{} For $(\lambda,r)\in P^+\times\bz$, the  global Weyl module $W(\lambda,r)$ is generated as a $\lie g[t]$--module by an element $w_{\lambda,r}$ with relations:\begin{gather*} \lie n^+[t]w_{\lambda,r}=0,\quad (x_{i}^-)^{\lambda(h_{i})+1}w_{\lambda,r}=0,\quad
 hw_{\lambda,r}=\lambda(h)w_{\lambda,r},
 \end{gather*} where $i\in I$ and $h\in\lie h$. The following result can be found in \cite{CPweyl} (see also \cite{CFK}).
 \begin{prop} \label{globweyl}For $(\lambda,r)\in P^+\times\bz$, we have that   $W(\lambda,r)$ is an indecomposable object of ${\cal I}$ and $\wt W(\lambda,r) =\wt V(\lambda,r)$. Further,
 \begin{enumerit}
   \item[(i)] $W_{\loc}(\lambda,r)$ is a quotient of $W(\lambda,r)$ and  $V(\lambda,r)$ is the unique irreducible quotient of $W(\lambda,r)$.
        \item[(ii)] $W(0,r)\cong\bc$ and  if $\lambda\ne 0$, the modules $W(\lambda,r)$ are infinite-dimensional.
 \item[(iii)] We have, $$\ch_{\gr}W(\lambda,r)= \ch_{\gr}V(\lambda,r)+\sum_{s>r}\sum_{\mu\le\lambda}\dim\Hom_{\lie g}(W\lambda,r)[s]: V(\mu))\ch_{\gr}V(\mu,s),$$ and $\{\ch_{\gr} W(\lambda,r): (\lambda,r)\in P^+\times\bz\}$  is a linearly independent subset of $\bz[P][[u, u^{-1}]]$.\end{enumerit}\hfill\qedsymbol\end{prop}

\subsection{}
  We say that $M\in\Ob\cal I$ admits a filtration by global Weyl modules if there exists a decreasing family of submodules $$M=M_0 \supset M_1\supset \cdots,\qquad \bigcap_k M_k=\{0\},$$  such that\begin{gather*} M_k/M_{k+1}\cong \bigoplus_{(\lambda,r)\in P^+\times\bz}W(\lambda,r)^{m_k(\lambda,r)}, \end{gather*}
 for some choice of $m_k(\lambda,r)\in \bz_+$.
Since $\dim M[r]_\lambda<\infty$ for all $(\lambda,r)\in P^+\times\bz$, we see that if $M$ has a  filtration by global Weyl modules,   then $m_k(\lambda,r)=0$ for all but finitely many $k$.
Further, we have $$\ch_{\gr}M= \sum_{k\ge 0}\ch_{\gr}M_k/M_{k+1}=
\sum_{(\lambda,r)\in\bz}\left(\sum_{k\ge 0}m_k(\lambda,r)\right)\ch_{\gr} W(\lambda,r).$$   Proposition \ref{globweyl}(iii) now implies that   the filtration multiplicity  $$[M: W(\lambda,r)]=\sum_{k\ge 0}m_k(\lambda,r),
$$ is well -defined and  independent of the choice of the filtration.

\subsection{}\label{proj2} The category $\cal I$ contains the projective cover of a simple object. For $(\lambda,r)\in P^+\times\bz$, set
 \begin{equation}\label{proj} P(\lambda,r)=\bu(\lie g[t])\otimes_{\bu(\lie g)} V(\lambda,r).\end{equation} Note that \begin{gather*} P(\lambda,r)[r]\ \cong_{\lie g} V(\lambda) ,\qquad  P(\lambda, r)[s]\ =0 \ \ \ s<r.\end{gather*}
The following  was proved in \cite[Proposition 2.1]{CG}.
\begin{prop}\label{pdefrel} For $(\lambda,r)\in P^+\times\bz$,  the object $P(\lambda,r)$ is generated by the element $p_{\lambda,r}=1\otimes v_{\lambda}$ with defining relations: $$\lie n^+p_{\lambda,r}=0,\ \ h p_{\lambda,r}=\lambda(h)p_{\lambda,r},\ \ (x_i^-)^{\lambda(h_i)+1} p_{\lambda,r}=0,$$ and is the  projective cover in $\mathcal I$ of  $V(\lambda, r)$.  Moreover, if $M\in\Ob{\cal I}$ then $$\Hom_{\cal I}(P(\lambda,r), M)\cong\Hom_{\lie g}(V(\lambda), M[r]).$$ \hfill\qedsymbol
\end{prop}

\subsection{} {\em From now on we fix an enumeration $\lambda_0,\lambda_1,\cdots,\lambda_k,\cdots$ of $P^+$ satisfying: $$\lambda_r-\lambda_s\in Q^+\implies r\ge s.$$}
We shall need the following result. Versions of this  have been proved  in the literature (see \cite{CFK} for instance). But  we include a proof here since we need it in this precise form for this paper.
\begin{lem} For  $k\ge 0$, the global Weyl module $W(\lambda_k,r)$ is the quotient of $P(\lambda_k,r)$ obtained by imposing the single additional relation $\lie n^+[t] p_{\lambda_k,r}=0.$ Equivalently, $W(\lambda_k,r)$ is the maximal quotient of $P(\lambda_k,r)$ whose weights lie in $\cup_{s=0}^k \lambda_s-Q^+$.  \end{lem}
\begin{pf}  The first statement is obvious from the defining relations of $W(\lambda_k,r)$ and $P(\lambda_k,r)$. For the second, let $$\tilde W =P(\lambda_k,r)/\sum_{s>k}\bu(\lie g[t])P(\lambda_k,r)_{\lambda_s}.$$ Clearly $$\wt\tilde W\subset\bigcup_{s=0}^k \lambda_s-Q^+,$$ and $\tilde W$ is the maximal quotient with this property. Let $\tilde w\in\tilde W$ be the image of $p_{\lambda_k,r}$. Since $\wt W(\lambda_k,r)\subset \lambda_k-Q^+$ it follows that $W(\lambda_k,r)$ is  a quotient of $\tilde W$ via a morphism which maps $\tilde w\to w_{\lambda_k,r}$. The element   $w'=(x_i^+\otimes t^s)\tilde w$ has weight $\lambda_k+\alpha_i>\lambda_k$. If it is non--zero in $\tilde W$ then it would follow from the representation theory of $\lie g$ that $\tilde W_{\lambda_s}\ne 0$ for some $s>k$ which is a contradiction.  Hence $\lie n^+[t]\tilde w= 0$ and there exists a well--defined surjective morphism $W(\lambda_k,r)\to \tilde W$ sending $w_{\lambda,r}\to \tilde w$ proving that $W(\lambda,r)\cong \tilde W$ as required.
\end{pf}

\subsection{} The main result of this paper is the following. It was conjectured in \cite{BCM} for all $\lie g$ and proved there in the case of  $\lie{sl}_2$.
\begin{thm} \label{bgg} Assume that $\lie g$ is of type $\lie{sl}_{n+1}$. For $(\lambda,r)\in P^+\times\bz$, the module $P(\lambda,r)$ has a filtration by global Weyl modules and $$[P(\lambda,r): W(\mu,s)]= [W_{\loc}(\mu,r): V(\lambda,s)].$$
\end{thm}
\begin{rem} We remark here that if $A$ is any graded commutative associative algebra, then  the Lie algebra $\lie g\otimes A$ and the category $\cal I$ can be  defined in the obvious way. The global and local Weyl modules and the projective modules have their analogs and  hence one could ask if Theorem \ref{bgg} remains true in this case. The graded characters of the local and global Weyl modules which play a crucial role in our paper are not known in this generality. However, the first step of the proof of the theorem which is Proposition \ref{red1} below does go through verbatim.\end{rem}

\subsection{}   The proof of the theorem is in two steps. The first step is the following.
\begin{prop}\label{red1} Let $\lie g$ be an arbitrary simple Lie algebra and let $M\in\Ob{\cal I}$ be such that $M[r]=0$ for all $r<<0$. There exists a decreasing filtration $$M=M_0 \supset M_1\supset \cdots,\qquad \bigcap_k M_k=\{0\},$$  and surjective morphisms $$\varphi_k: \bigoplus_{r\in\bz_+} W(\lambda_{k},r)^{m(k,r)}\longrightarrow M_k/M_{k+1}\to 0, \quad k\ge 0 $$ where $ m(k,r)=\dim\Hom_{\cal I }(M, W_{\loc}(-w_0\lambda_{k},-r)^*)$.
\end{prop}
The proposition will be proved in the next section.

\subsection{}  The second step in the proof of the theorem is the following.
\begin{prop} \label{red2} Assume that $\lie g$ is of type $\lie{sl}_{n+1}$.  We have,
\begin{gather}\mathbb H(P(\lambda,0))=\sum_{k\ge 0}\sum_{ r\in\bz_+}[W_{\loc}(\lambda_{k},0): V(\lambda,r)]\mathbb H(W(\lambda_{k},r))\\=\sum_{k\ge 0}\left(\sum_{r\ge 0}[W_{\loc}(\lambda_{k},0): V(\lambda,r)]u^r\right)\mathbb H(W(\lambda_{k},0).\end{gather}
\end{prop}
This proposition is proved in the last two sections of this paper.

 \subsection{} Observe that we can apply Proposition \ref{red1} to $P(\lambda,r)$.  Using the following equalities which follow from Proposition \ref{proj2} and standard properties of duals and grade shift operators, we get
\begin{gather*}m(k,r)= \dim\Hom_{\cal I}(P(\lambda,0), W_{\loc}(-w_0\lambda_k,-r)^*) =\dim\Hom_{\lie g}(V(\lambda),  W_{\loc}(-w_0\lambda_k,-r)^*[0])\\ =\dim\Hom_{\lie g}(  W_{\loc}(-w_0\lambda_k,-r)[0],V(-w_0\lambda))=\dim\Hom_{\lie g}(  W_{\loc}(\lambda_k,0)[r],V(\lambda))\\ =
[W_{\loc}(\lambda_k,0): V(\lambda,r)].\end{gather*}
 Hence for $\ell\ge 0$, we have $$\dim P(\lambda,0)[\ell ]\le\sum_{k,r\ge 0}[W_{\loc}(\lambda_{k},0): V(\lambda,r)]\dim W(\lambda_{k},r)[\ell].$$ Proposition \ref{red2} implies that for $\lie{sl}_{n+1}$,  equality holds and hence the surjective maps $\varphi_k$ are actually isomorphisms for all $k\ge 0$. This proves Theorem \ref{bgg}.

 \subsection{} In the last section, we  also establish the analog of Theorem \ref{bgg} in certain subcategories of $\cal I$.  Given $k\ge 0$, let $\cal I_{>}^k$ be the full subcategory of $\cal I_{>}$ consisting of objects $M$ such that $$\wt M\subset\bigcup_{s=0}^k\lambda_s-Q^+.$$ The modules $V(\lambda_s,r)$, $W_{\loc}(\lambda_s,r)$ and $W(\lambda_s,r)$ are objects of  $\cal I_{>}^k$ for all $s\le k$ and $r\in\bz$. Let $P^k(\lambda_s,r)$ be the maximal quotient of $P(\lambda_s,r)$ which lies in $\cal I_{>}^k$. Then $P^k(\lambda_s.r)$ is the projective cover of $V(\lambda_s,r)$ in $\cal I_{>}^k$.
 \begin{thm}\label{bggtr} Assume that $\lie g$ is of type $\lie{sl}_{n+1}$. Let $s,k\in\bz_+$ with $s\le k$. The object $P^k(\lambda_s,r)$ has a finite filtration by global Weyl modules, and  $$[P^k(\lambda_s,r): W(\lambda_\ell, p)]= [W_{\loc}(\lambda_\ell,r): V(\lambda_s,p)].$$
 \end{thm}

\section{The Canonical filtration} Let $\cal I_{>}$ be the full subcategory of $\cal I$ consisting of objects
$M$ whose grades are bounded below, i.e., there exists $r\in\bz$ (depending on $M$) with $M[p]=0$ for all $p< r$. It follows from Section \ref{proj} that $P(\lambda,r)\in\Ob\cal I_{>}$ for all $\lambda\in P^+$.

\subsection{} We begin this section with the following proposition which summarizes the properties of the  duals of the projective, global and local Weyl modules.
\begin{prop}\label{dual} For $(\lambda_k,r)\in P^+\times\bz$, set $$I(\lambda_k,r)= P(-w_0\lambda_k,-r)^*.$$
\begin{enumerit}
\item[(i)] $I(\lambda_k,r)$ is an injective object of $\cal I$  with a unique irreducible submodule which is isomorphic to $V(\lambda_k,r)$.
    \item[(ii)] The maximal submodule  of $I(\lambda,r)$ whose weights are in the union of cones $\lambda_s-Q^+$, $0\le s\le k$ is isomorphic to $W(-w_0\lambda,-r)^*$.
        \item[(iii)] $W_{\loc}(-w_0\lambda_k,-r)^*$ is isomorphic to the maximal submodule $M$  of $I(\lambda,r)$ satisfying $$\wt M\subset\bigcup_{s=0}^r\lambda_s-Q^+,\qquad M[s]_{\lambda_k}\ne 0\implies s=r.$$
            \end{enumerit}\hfill\qedsymbol
\end{prop}
\noindent By abuse of notation we shall freely identify $V(-w_0\lambda_k,-r)$  with  its isomorphic copy in $W_{\loc}(\lambda_k,r)^*$, similar remarks apply for the corresponding submodules of $I(\lambda,r)$.

 \subsection{} Given $M\in\Ob\cal I_{>}$, let $$M_k =\sum_{s\ge k}\bu(\lie g[t]M_{\lambda_s},\qquad s,k\in\bz_+.$$ Clearly $M_0=M$,  $M_k\in\Ob\cal I_{>}$ for all $k\ge 0$ and \begin{gather}\label{onek} M_k/M_{k+1}=\bu(\lie g[t])(M_k/M_{k+1})_{\lambda_k},\\ \nonumber \\ \label{2k} \lie n^+[t] \left(M_k/M_{k+1}\right)_{\lambda_k}=0,\qquad \ (h-\lambda_k(h))\left(M_k/M_{k+1}\right)_{\lambda_k}=0,\ \ h\in\lie h.\end{gather}
 We claim that $$\bigcap_{k\in\bz_+}M_k=\{0\},$$ and call  $M=M_0\supset M_1\cdots$  the {\em canonical filtration}  of $M\in\Ob{\cal I}_>$.

 \vskip12pt

\iffalse For the claim, observe that if $\wt M\cap P^+$ is finite, then $M_k=\{0\}$ for some $k$ and there is nothing to prove. Otherwise, we have $M_{\lambda_k}\ne 0$ for infinitely many $k\in\bz_+$.  Suppose that $ M[r]\ne 0$ for some $r\in\bz_+$.\fi For the claim, note that  since $M[s]=0$ for all $s<<0$, we have $$|\left(\bigcup_{p\le r}\wt M[p]\right) \bigcap P^+|<\infty.$$ Hence, we can
 choose  $k_0\in\bz_+$ such that \iffalse $$ \  \bigcup_{p\le r}\wt M[p]\subset\bigcup_{s=0}^{k_0-1}\lambda_s-Q^+,\qquad {\rm{i.e.}}\ \ M_{\lambda_s}[p]=0,\ \ {\rm{if }}\ s\ge k_0\ {\rm{and}} \  p\le r.$$ In other words, we have \fi $$M_{\lambda_s}\subset\bigoplus_{p>r}M[p],\ \ s\ge k_0.$$  Using the definition of $M_{{k_0}}$ we now get that $M_{k_0}[r]=\{0\}$  which establishes the claim.

\subsection{} \begin{lem} For $M\in\Ob{\cal I}_>$, and $k,r\in\bz_+$, we have an isomorphism of vector spaces,\begin{equation}\label{chineq}\Hom_{\lie g}(M_k/M_{k+1}, V(\lambda_k,r))\cong \Hom_{\cal I}( M, W_{\loc}(-w_0\lambda_k, -r)^*).\end{equation}
\end{lem}
 \begin{pf}
   Let $\iota_k: M_k/M_{k+1}\to M/M_{k+1}$ and  $\iota_{\lambda_k,r}: V(\lambda_k,r)\to I(\lambda_k,r)$ be the canonical injective morphisms. Given $\psi\in \Hom_{\lie g}(M_k/M_{k+1}, V(\lambda_k,r))$, let $\tilde\psi: M/M_{k+1}\to I(\lambda_k,r)$ be the map such that $$\tilde\psi\iota_k=\iota_{\lambda,r}\psi.$$ Since \begin{gather*}\wt M/M_{k+1}\subset \bigcup_{s=0}^k\lambda_s-Q^+,\qquad (M/M_{k+1})_{\lambda_k}\cong (M_k/M_{k+1})_{\lambda_k} \end{gather*} and $$\psi((M_k/M_{k+1})[p]_{\lambda_k})=0, p\ne r, $$it follows from Proposition \ref{dual}(iii) that $$\Im\tilde\psi\subset W_{\loc}(-w_0\lambda_k,-r)^*.$$ If $\pi:M\to M/M_{k+1}$ is the canonical projection, we see now that the assignment $\psi\to\pi.\tilde\psi$ is an injective linear map $\Hom_{\lie g}(M_k/M_{k+1}, V(\lambda_k,r))\longrightarrow \Hom_{\cal I}( M, W_{\loc}(-w_0\lambda_k, -r)^*)$.

 For the converse, let $\psi\in\Hom_{\cal I}( M, W_{\loc}(-w_0\lambda_k, -r)^*)$. Observe that $\psi(M_{k+1})=0$ , since $\wt W_{\loc}(-w_0\lambda_k, -r)^*)\subset\lambda_k-Q^+$. Hence  we have a non--zero map $ M/M_{k+1}\to  W_{\loc}(-w_0\lambda_k, -r)^*$ and we let $\tilde\psi$ be the restriction of this map to $M_k/M_{k+1}$.
 Since $V(\lambda_k,r)$ is the  unique irreducible  submodule of $W_{\loc}(-w_0\lambda_k, -r)^*$, we have
   $$\Im\psi\cap V(\lambda_k,r)\ne 0 \ \qquad \Im(M/M_{k+1})\cap V(\lambda_k,r)\ne 0.$$  Since $M_{\lambda_k}=(M_k)_{\lambda_k}$ it now follows that $\Im\tilde\psi =V(\lambda_k,r),$
    and the assignment $$\psi\to\tilde\psi, \qquad \Hom_{\cal I}( M, W_{\loc}(-w_0\lambda_k, -r)^*)\to \Hom_{\lie g}(M_k/M_{k+1}, V(\lambda_k,r)),$$ is an injective map and the Lemma is proved.

 \end{pf}
\subsection{}\label{enump}   Let  ${\rm{head}}(M)$ be the  maximal semi--simple quotient of $M$ and $\boh: M\to {\rm{head}}(M)$ be the corresponding map.  If $r$ is minimal such that $M[r]\ne 0$, then $\oplus_{p>r}M[p]$ is a proper submodule of $M$. The corresponding quotient is semisimple and isomorphic to $\ev_0M[r]$ and hence ${\rm{head}}(M)\ne 0$. Moreover, by Lemma \ref{irred}, we have
 $${\rm{head}}(M)= \bigoplus_{(\lambda,\ell)\in P^+\times\bz}  V(\lambda,\ell)^{\oplus\dim\Hom_{\cal I}(M, V(\lambda,\ell))}.$$  The map $\boh$ lifts to a   surjective  map \begin{equation}\label{projmap}\tilde\boh: \bigoplus_{(\lambda,\ell)\in P^+\times \bz} P(\lambda,\ell)^{\oplus\dim\Hom_{\cal I}(M, V(\lambda,\ell))}\longrightarrow M\to 0.\end{equation} The fact that a lift $\tilde\boh$ of $\boh$ exists is obvious since the $P(\lambda,\ell)$ are projective and we have surjective maps $P(\lambda,\ell)\to V(\lambda,\ell)\to 0$ sending $p_{\lambda,\ell}\to 1\otimes v_{\lambda,\ell}$. If $M'=M/\Im\tilde\boh$ is non--zero then, ${\rm{head}}(M')\ne 0$ and hence is a semisimple quotient of $M$ as well. But this  contradicts the fact that ${\rm{head}}(M)$ is the maximal semisimple quotient of $M$ and the fact that $$\Im\tilde\boh\supset {\rm{head}}(M).$$
\subsection{}  For $r\in\bz$, set  $m(k,r)=\dim\Hom_{\lie g}(M_k/M_{k+1}, V(\lambda_k,r))$ and notice that $${\rm{head}}(M_k/M_{k+1})\cong \bigoplus_{r\in\bz}  V(\lambda_k,r)^{m(k,r)}.$$. Using Corollary \ref{proj2} we see that  the map $$\bigoplus_{r\in\bz} P(\lambda_k,r)^{m(k,r)}\to M_k/M_{k+1}\to 0 ,$$ defined  in \eqref{projmap} factors through to $$\bigoplus_{r\in \bz}W(\lambda_k,r)^{\oplus m(k,r)}\longrightarrow M_k/M_{k+1}\to 0.$$
Proposition \ref{red1} now follows by using  Lemma \ref{chineq}.

\section{A combinatorial interlude}

\subsection{} Let $\{x_: 1\le j\le r\}$ be a set of indeterminates. The symmetric group $S_r$ acts naturally on the polynomial ring $\bz[x_1,\cdots, x_r]$ and we let $\Lambda_r$ be the corresponding ring of invariants. Set  $|x|=x_1\cdots x_r$ and denote by $\Lambda_r'$ the localization of $\Lambda_r$ at $|x|$. Equivalently $\Lambda_r'$ is the ring of invariants for the action of $S_r$ on the integer valued ring of Laurent polynomials in the $x_j$, $1\le j\le r$.
  Let $\hat{\Lambda}_r$ be the ring of all symmetric power series in $x_1,\cdots,x_r$, thus elements of $\hat{\Lambda}_r$ are of the form $\sum_{\ell\ge 0} p_\ell$ where $p_\ell\in \Lambda_r$ is homogeneous of degree $\ell$.

 \subsection{} Given any $f\in\bz[x_1^\pm,\cdots ,x_r^{\pm 1}]$, let $[f]_0$ be the constant term of $f$ and  let  $f^*$ be given by $f^*(x_1,\cdots, x_r)= f(x_1^{-1},\cdots ,x_r^{-1})$.
The  Macdonald inner product  $(\ ,\ ):\Lambda_r'\times\Lambda_r'\to \bz$  is,
$$(f,g)=\frac{1}{r!}\left[fg^*\prod_{1\le i<j\le r} (1-\frac{x_i}{x_j})(1-\frac{x_j}{x_i})\right]_0,\ \  f,g\in\Lambda_r'.$$ Observe that since $(f,g)=0$ if $f,g\in\Lambda_r$ are homogenous of different degree we can extend this naturally to a map $\Lambda_r\times\hat{\Lambda}_r\to \bz$.
The following is obvious (and, apparently, well-known).

\begin{lem}
\label{le:ort}
\begin{enumerit}
\item[(i)]
 For any $f,g,h\in \Lambda'_r$ one has
$$(fg\ ,\ h)=(f\ ,\ g^* h).$$
\item[(ii)] For any $f\in \Lambda_r$ and $\ell_1,\ldots,\ell_r\in {\bold Z}$ one has:
$$\left(f,\frac{1}{\prod\limits_{1\le i,j\le r} (1-x_i \ell_j)}\right)=f(\ell_1,\ldots,\ell_r).$$
\end{enumerit}\hfill\qedsymbol
\end{lem}

\subsection{} Let ${\rm{Par}}(r)$ be the set of all partitions $\xi=(\xi_1\ge\cdots\ge\xi_r\ge 0)$ with at most $r$ parts. Given  $\xi=(\xi_1\ge\xi_2\ge\cdots\xi_r\ge 0)\in {\rm{Par}}(r)$  let $$m_\xi=\sum_{\sigma\in S_r}x_{\sigma(1)}^{\xi_1}\cdots x_{\sigma(r)}^{\xi_r}\in\Lambda_r.$$ The set $\{m_\xi: \xi\in{\rm{Par}}(r)\}$  is  an integral basis of $\Lambda_r$, called the {\em symmetrized monomial basis}.

The basis OF $\Lambda_r$  consisting of Newton polynomials is given  as follows. For $0\le j\le r$ and a partition $\xi=(\xi_1\ge\cdots\ge\xi_r\ge 0)$ set,  \begin{gather*}p_j = x_1^j + x_2^j + \cdots +x_r^j,\qquad p_\xi = p_{\xi_1} \dots p_{\xi_r}.\end{gather*}

 The basis  of Schur functions $s_\xi$ is defined as follows.
Given  $ \bom=(m_1,\cdots,m_r)\in\bn^r$ let
$$d_\bom = \det
\begin{pmatrix}
x_1^{m_1} & \dots & x_r^{m_1}\\
\vdots & \vdots  & \vdots \\
x_1^{m_r} & \dots & x_r^{m_r}
\end{pmatrix}
$$
Then it can be shown that for a partition $\xi$ the polynomial $d_{(\xi_1+r-1,\xi_2+r-2, \dots, \xi_r)}$
is divisible by $d_{(r-1,r-2, \dots, 1,0)}$ and the ratio is the Schur function $s_\xi$.
{{If $\xi=(\xi_1\ge\cdots\ge\xi_r)$,  we have $$ |x|s_\xi(x_1,\ldots,x_n)=s_{\tilde\xi}(x_1,\ldots,x_n),\qquad  \tilde\xi=(\xi_1+1\ge\cdots\ge\xi_r+1).$$ In particular, this means that if $\xi\in {\rm{Par}}(r)$ and $\lambda\in {\rm{Par}}(r-1)$ is such that $\lambda_s=\xi_s-\xi_r$, $1\le s\le r-1$, then \begin{equation}\label{multischur}s_{\xi}(x_1,\cdots, x_r)=|x|^{\xi_r}s_\lambda(x_1,\cdots,x_r).\end{equation}
Moreover, it is well--known that  the elements
$$s_{\lambda,\ell}=|x|^\ell s_\lambda(x_1,\ldots,x_r),\qquad
\lambda\in {\rm Par}(r-1),\ \ \ell\in {\bold Z},$$ form an orthonormal (with respect to $(\ ,\ )$) ${\bold Z}$-linear basis of $\Lambda'_r$,
 and
$$\Lambda'_r=\bigoplus_{\ell\in {\bold Z}}  \Lambda^0_r\cdot |x|^\ell,\qquad \Lambda_r=\bigoplus_{\ell\ge 0} \Lambda^0_r\cdot |x|^\ell,$$
where $\Lambda_r^0$ is the ${\bold Z}$-linear span of  $\{s_\lambda(x_1,\ldots,x_r):\lambda\in {\rm Par}(r-1)\}$.

}}

\subsection{}
Define elements  $R_r,R'_r\in \hat \Lambda_r$ by
\begin{gather*}R_r=\frac{1}{(1-x_1)^r \cdots (1-x_r)^r},\\ \\
R'_r=\frac{1-x_1\cdots x_r}{(1-x_1)^r \cdots (1-x_r)^r}.\end{gather*}
Note that $R_r$ (resp. $R'_r$) can be viewed as the non-singular part of the character of the regular
representation of $U_r$ (resp. $SU_r$).

\begin{lem}\label{cau}
We have $$R_r
= \sum_{\xi \in {\rm{Par}}(r)} s_\xi(1,\ldots,1) s_\xi(x_1,\ldots, x_r)
,$$
$$R'_r
= \sum_{\lambda \in {\rm{Par}}(r-1)} s_\lambda(1,\ldots,1) s_\lambda(x_1,\ldots, x_r)
.$$
\end{lem}
\begin{proof} Let $y_1,\cdots,y_r$ be another set of indeterminates. Then, setting $y_1=\cdots=y_r=1$ in the Cauchy identity \cite[I, (4.3)]{Mac}:
$$\frac{1}{\prod\limits_{1\le i,j\le r} (1-x_i y_j)} = \sum_{\xi\in {\rm Par}(r)} s_\xi(x_1,\ldots,x_r) s_\xi(y_1,\ldots,y_r)$$ gives the  first identity of the Lemma. To prove the second one,
we use \eqref{multischur} to get
% hence $s_\xi(1,\ldots,1)=s_{\xi+(1,\ldots,1)}(1,\ldots,1)$. Therefore,
\begin{eqnarray*}\frac{1}{\prod\limits_{1\le i,j\le r} (1-x_i y_j)}
&=& \sum_{\xi \in {\rm{Par}}(r)} s_\xi(x_1,\ldots,x_r) s_{\xi}(y_1,\ldots, y_r)\\
&=&\sum_{\lambda\in {\rm{Par}}(r-1)}s_\lambda(x_1,\ldots,x_r) s_\lambda(y_1,\ldots,y_r)\sum_{\ell\in\bz_+} |x|^\ell|y|^\ell \\
&=&\sum_{\lambda\in {\rm{Par}}(r-1)}\frac{s_\lambda(x_1,\ldots,x_r) s_\lambda(y_1,\ldots, y_r)}{1-|x|  |y|}.\end{eqnarray*}
Hence,
$$\frac{1-x_1\cdots x_ry_1\cdots y_r}{\prod\limits_{1\le i,j\le r} (1-x_i y_j)}=\sum_{\lambda\in {\rm{Par}}(r-1)} s_\lambda(x_1,\ldots,x_r) s_\lambda(y_1,\ldots, y_r).$$
Now  setting $y_1=\cdots=y_r=1$ completes the proof of the second identity.
\end{proof}

\subsection{} We now prove,  \begin{prop}\label{mult}
 Let  $\lambda\in {\rm Par}(r-1)$, $\ell\in\bz_+$. If $f\in\Lambda_r$ is such that    $\ell\ge \deg(f)$, then
$$(s_{\lambda,\ell}\ , \ fR_r)=f(1,\ldots,1)s_\lambda(1,\ldots,1)$$
\end{prop}
\begin{proof} {{By Lemma \ref{le:ort} (i), we have $$(s_{\lambda,\ell}\ ,\ f R_r)=(s_{\lambda,\ell}f^*
\ , \ R_r) .$$ Since $\ell \ge \deg f$, we have $s_{\lambda,\ell}f^*=s_{\lambda,\ell-\deg(f)}(|x|^{\deg(f)}f^*)\in \Lambda_r$ and we can now use Lemma \ref{le:ort}(ii) with  $\ell_1=\cdots=\ell_r=1$ to get
$$(s_{\lambda,\ell}f^*
\ , \ R_r) =(s_{\lambda,\ell}f^*)(1,\cdots ,1)= s_{\lambda}(1,\cdots ,1)f(1,\cdots ,1),$$ as required.}}

\end{proof}

\subsection{} Set, $$(a;q)_\infty:=\prod_{i=0}^\infty(1-aq^i).$$ Our goal is to find an asymptotic expansion of the element
$$Q:=\frac{(q x_1\cdots x_r; q)_\infty}{(x_1;q)_\infty^r \cdots(x_r;q)_\infty^r}=\sum_{\lambda\in {\rm Par}(r-1),\ell\ge 0}\phi_{\lambda,\ell}(q) s_{\lambda,\ell}, $$
in $\hat \Lambda_r$, where \begin{equation}
\label{a} \phi_{\lambda,\ell}(q)=(s_{\lambda,\ell}\ ,\ Q) =\sum_{s\ge 0} a_{\lambda,\ell}^s q^s \in {\bold Z}[[q]].\end{equation} Define integers $c_m$ by requiring, $$\frac{1}{(q;q)^{r^2-1}_\infty}=\sum_{m\ge 0} c_m q^m.$$
\begin{prop}\label{lim} For  $\lambda\in {\rm Par}(r-1)$ and  $\ell\ge 0$  we have
$$a_{\lambda,\ell}^m= c_ms_\lambda(1,\ldots,1),\ \ {\rm{if}}\ \  \ell \ge rm. $$
\end{prop}

\begin{proof} Setting, $$Q'=\frac{(q x_1\dots x_r; q)_\infty}{(qx_1;q)_\infty^r \dots(qx_r;q)_\infty^r},$$ it is clear that
$$Q=\frac{(q x_1\cdots x_r; q)_\infty}{(x_1;q)_\infty^r \cdots(x_r;q)_\infty^r}= R_r\cdot Q'.$$
Writing   $Q'=\sum_{m\ge 0}Q'_mq^m$ as a power series in $q$ with coefficients in $\Lambda_r$, we get from \eqref{a} that
$$\phi_{\lambda,\ell}(q)=(s_{\lambda,\ell},Q)=(s_{\lambda,\ell},R_r\sum_{m\ge 0} Q'_m q^m)=\sum_{m\ge 0}a_{\lambda,\ell}^m q^m,\qquad i.e.,a_{\lambda,\ell}^m=(s_{\lambda,\ell},R_r Q'_m). $$
Since  $\deg(Q'_m)\le rm$, Proposition \ref{mult} implies that
$$a_{\lambda,\ell}^m=Q'_m(1,\ldots,1)s_\lambda(1,\ldots,1),\qquad {\rm{if}}\ \ \ell\ge mr. $$ The proposition follows by noticing that if we set $x_1=\cdots =x_r=1$, we have
$$\frac{1}{(q;q)^{r^2-1}_\infty}=Q'(1,\ldots,1)= \sum_{m\ge 0} Q'_m(1,\ldots,1) q^m=\sum_{m\ge 0} c_m q^m.$$

\end{proof}

\subsection{} {{

For $\lambda\in {\rm Par}(r-1)$ and  $\ell\ge 0$ define power series $\psi_{\lambda,\ell}=\sum_{m\ge 0}b_{\lambda,\ell}^mq^m\in {\bold Z}[[q]]$  by:
$$\psi_{\lambda,\ell}(q)=\phi_{\lambda,\ell}(q)-\phi_{\lambda,\ell-1}(q),\qquad  b_{\lambda,\ell}^m=a_{\lambda,\ell}^m-a_{\lambda,\ell-1}^m,$$ where we adopt the convention that $\phi_{\lambda,-1}=0$.  Clearly,$$\phi_{\lambda,\ell}(q) =  \sum\limits_{k=0}^\ell  \psi_{\lambda,k}(q),$$ and moreover, we see from Proposition \ref{lim} that  $b_{\lambda,\ell}^m=0$ when $\ell>mr$. This means that
$$\sum_{\ell\ge 0} b_{\lambda,\ell}^m<\infty,$$ and hence
$$\sum_{\ell\ge 0}\psi_{\lambda,\ell}(q)=\sum_{\ell\ge 0}\sum_{m\ge 0} b_{\lambda,\ell}^mq^m=\sum_{m\ge 0}\left(\sum_{\ell\ge 0} b_{\lambda,\ell}^m\right)q^m,$$ is a well--defined element of $\bz[[q]]$.
 Using Proposition \ref{lim} again,  we see  that $$\sum_{k\ge 0 }\psi_{\lambda,k}=\lim_{\ell\to\infty}\phi_{\lambda,\ell}=\sum_{m\ge 0}c_ms_\lambda(1,\cdots,1)q^m= \frac{s_\lambda(1,\cdots ,1)}{(q;q)^{r^2-1}_\infty}.$$ Together with the fact that \begin{eqnarray*}\sum_{\lambda\in {\rm Par}(r-1)}\sum_{ \ell\ge 0}\psi_{\lambda,\ell}(q)s_{\lambda,\ell}&=\sum_{\lambda\in {\rm Par}(r-1)}\sum_{ \ell\ge 0}\phi_{\lambda,\ell}(q)(s_{\lambda,\ell}-s_{\lambda,\ell+1})\\ & =
\sum_{\lambda\in {\rm Par}(r-1)}\sum_{ \ell\ge 0}\phi_{\lambda,\ell}(q)(1-|x|)s_{\lambda,\ell},\end{eqnarray*} we have now proved,
\begin{prop}\label{lhs} We have an equality of symmetric power series,
 $$\frac{(x_1\cdots x_r; q)_\infty}{(x_1;q)_\infty^r \cdots(x_r;q)_\infty^r}=\sum_{\lambda\in {\rm Par}(r-1)}\sum_{ \ell\ge 0}\psi_{\lambda,\ell}(q)s_{\lambda,\ell},$$ where
$\psi_{\lambda,\ell}(q)\in q^{\lfloor \frac{\ell}{r}\rfloor}\cdot {\bold Z}[[q]]$ and
$$\sum_{\ell\ge 0}  \psi_{\lambda,\ell}(q) = \frac{s_\lambda(1,\ldots,1)}
{(q;q)^{r^2-1}_\infty}.$$\hfill\qedsymbol
\end{prop}
}}

\subsection{} Let $q,t$ be indeterminates and let  $\bq(q,t)$ be the field of rational functions in $q$ and $t$ over the field $\bq$  of rational numbers.
The {\em Macdonald scalar product} on $\Lambda_r(q,t)=\Lambda_r\otimes\bq(q,t)$, is defined on the Newton polynomials by
$$\left< p_\xi, p_\psi\right>_{q,t} = \delta_{\xi,\psi} \prod_{i=1}^r  i^{n_i} n_i! \prod_{s=1}^{\ell(\xi)} \frac{1-q^{\xi_s}}{1-t^{\xi_s}},$$
for $\xi,\psi\in {\rm Par}(r)$,  where $n_i=|\{k:\xi_k=i\}|$ and $\ell(\xi)$ is the number of non--zero parts of $\xi$.
The Macdonald polynomials $P_\xi(x; q,t)$ in $x=(x_1,\ldots,x_r)$  is the orthonormal basis of $\Lambda_r(q,t)$ obtained by applying the Gram--Schmidt process to the lexicographically ordered  basis of monomial symmetric functions. Thus, $$P_\xi(x;q,t)= m_\xi+\sum_{\psi<\xi}u_{\xi,\psi}m_\xi(x),\ \ u_{\xi,\psi}\in\bq(q,t).$$

\begin{prop}\label{McCauchy} For $r\ge 1$ one has:
\begin{equation}\label{mcc}\frac{1}{\prod\limits_{1\le i,j\le r} (x_iy_j; q)_\infty}= \sum_{\xi\in{\rm{Par}}(r) } \frac{P_\xi(x; q,0) P_\xi(y; q,0)}{(q;q)_{\xi_1-\xi_2} \cdots
(q;q)_{\xi_{r-1}-\xi_r} (q;q)_{\xi_r}},\end{equation}
where $(a;q)_m:=\prod_{i=0}^m(1-aq^i)$ denotes the the (shifted) $q$-Pochhammer symbol.
\end{prop}
\begin{proof}
 It is shown in \cite[VI,(4.19)]{Mac}, that
\begin{equation}
\label{eq:Macort}
\prod_{i,j=1}^r \frac{(tx_iy_j; q)_\infty}{(x_iy_j; q)_\infty} = \sum_{\xi\in{\rm{Par}}(r) } b_\xi(q,t)  P_\xi(x;q,t)P_\xi(y;q,t),
\end{equation}
where $b_\xi(q,t) = \left< P_\xi(x,q,t),  P_\xi(x,q,t) \right>_{q,t}^{-1}$ is computed by the following recursive formula in \cite[VI (6.19)]{Mac},
 \begin{equation}\label{eq:Macrec} b_0(q,t)=1,\qquad
b_\xi(q,t)=b_{\xi'}(q,t)\prod_{i=1}^r  \frac{1-q^{\xi_i-1} t^{r+1-i}}{1-q^{\xi_i} t^{r-i}}
\end{equation}
for $\xi=(\xi_1\ge\cdots\ge \xi_r)\in {\rm Par}(r)$ with $\xi_r>0$, and $\xi'=(\xi_1-1\ge\cdots\ge \xi_r-1)$.

If we set
$t=0$  in the above formulas, we see that the left hand side of \eqref{eq:Macort} is precisely the left hand side of \eqref{mcc}.
Also, the recursion \eqref{eq:Macrec} simplifies:
$$b_\xi(q,0)=b_{\xi'}(q,0) \frac{1}{1-q^{\xi_i}}$$ and gives by induction:
$$b_\xi(q,0) =  (q;q)^{-1}_{\xi_1 - \xi_2} \cdots (q;q)^{-1}_{\xi_{r-1} - \xi_r}
(q;q)^{-1}_{\xi_{r}}.$$
The proposition is proved.
\end{proof}
\subsection{} \begin{prop} We have,
$$\frac{(x_1\cdots x_r;q)_\infty}{\prod\limits_{1\le j\le r} (x_j; q)_\infty}=\sum_{\lambda\in {\rm{Par}}(r-1)}
\frac{P_{\lambda}(1,\cdots,1; q,0) P_{\lambda}(x_1,\dots, x_r; q,0)}
{(q;q)_{\lambda_1-\lambda_2} \dots (q;q)_{\lambda_{r-2}-\lambda_{r-1}}(q;q)_{\lambda_{r-1}}}.$$
\end{prop}
\begin{pf} Using the fact that $$P_{\xi}(x; q,t) = |x|^{\xi_r}P_\lambda(x; q,t),\ \qquad \lambda=(\xi_1-\xi_r\ge\cdots\ge\xi_{r-1}-\xi_r\ge 0),$$ we get,
\begin{gather*}\sum_{\xi\in{\rm{Par}}(r) } \frac{P_\xi(x; q,0) P_\xi(y; q,0)}{(q;q)_{\xi_1-\xi_2} \cdots
(q;q)_{\xi_{r-1}-\xi_r} (q;q)_{\xi_r}}=\sum_{\ell\ge 0}\frac{|xy|^\ell}{(q;q)_\ell}\sum_{\lambda\in{\rm{Par}}(r-1) }\frac{P_\lambda(x; q,0) P_\lambda(y; q,0)}{(q;q)_{\lambda_1} \cdots
(q;q)_{\lambda_{r-1}} }\\ = \frac{1}{(|xy|;q)_\infty}\sum_{\lambda\in{\rm{Par}}(r-1) }\frac{P_\lambda(x; q,0) P_\lambda(y; q,0)}{(q;q)_{\lambda_1} \cdots
(q;q)_{\lambda_{r-1}} }, \end{gather*} where we have used the fact that for any indeterminate $a$, we have $$\sum_{k=0}^\infty\frac{a^k}{(q;q)_k}= \frac{1}{(a;q)_\infty}.$$ Setting $y_1=\cdots y_r=1$ and using Proposition \ref{McCauchy} completes the proof.
\end{pf}
\subsection{} The following is now an immediate consequence of Proposition \ref{lhs} and Proposition \ref{McCauchy}.
\begin{lem}\label{crucial} For $\mu\in {\rm{Par}}(r-1)$, write $$P_{\mu}(x_1,\dots, x_r; q,0) = \sum_{\ell\ge 0}\sum_{ \lambda\in {\rm{Par}}(r-1)} \eta^\mu_{\lambda,\ell} (q) s_{\lambda,\ell}(x_1,\dots, x_r).$$ Then,
$$\frac{s_\lambda(1,\ldots,1)}
{(q;q)^{r^2-1}_\infty}=\sum_{\mu\in {\rm{Par}}(r-1)}\sum_{\ell\ge 0}
\frac{\eta_{\lambda,\ell}^\mu(q) P_{\mu}(1,\cdots,1; q,0) }
{(q;q)_{\mu_1-\mu_2} \dots (q;q)_{\mu_{r-2}-\mu_{r-1}}(q;q)_{\mu_{r-1}}}.$$
\end{lem}

\section{Proof of Proposition \ref{red2}} The proof   involves putting together known results on the Hilbert series of the projective, global Weyl and local Weyl modules with the combinatorial identities which were established in the previous section. We will use the notation of the previous sections freely.

\subsection{} The Hilbert series  of $P(\lambda,0)$ is easily calculated by using the Poincare Birkhoff Witt theorem. Thus, we have $$P(\lambda,0)=\bu(\lie g[t])\otimes_{\bu(\lie g)} V(\lambda)\cong\bu(\lie g[t]_+)\otimes V(\lambda),$$ where $\lie g[t]_+=\lie g\otimes t\bc[t]$ and the isomorphism is one of vector spaces. In particular,
$$P(\lambda,0)[s]=\bu(\lie g[t]_+)[s]\otimes V(\lambda),\ \ {\rm{and}}\ \ \bu(\lie g[t]_+)[s]\cong S(\lie g[t]_+)[s],$$ where $S(\lie g[t]_+)$ is the symmetric algebra of $\lie g[t]_+$, and the isomorphism is again one of vector spaces.
Together with Proposition \ref{lhs}, we get
\begin{equation}\label{Ver}\mathbb H(P(\lambda,0))=\dim V(\lambda)\mathbb H(S(\lie g[t]_+))=\frac{\dim V(\lambda)}{(q;q)_{\infty}^{\dim{\lie g}}} =\sum_{\ell\ge 0}\psi_{\lambda,\ell}(q).\end{equation}
In the special case when $\lie g$ is of type $\lie{sl}_{r}$, it is well--known that $\dim V(\lambda)=s_\lambda(1,\cdots,1)$ where we identify $P^+$ with the set ${\rm{Par}}(r-1)$ by sending $\lambda=\sum_{i=1}^{r-1}\lambda_i\omega_i$ to the partition whose $j^{th}$ part is $\sum_{i=j}^{r-1}\lambda_j\omega_j$. Hence we have,
\begin{lem}\label{hproj} Suppose that $\lie g$ is of type $\lie{sl}_{r}$. Then,
$$\mathbb H(P(\lambda,0))= \frac{s_\lambda(1,\cdots,1)}{(q;q)_\infty^{r^2-1}}.$$\hfill\qedsymbol
\end{lem}
\subsection{} We now recall the relationship between the graded characters of global and local Weyl modules. This was established in \cite[Proposition 3.7]{BCM} using results proved in \cite{CL}, \cite{CPweyl},  \cite{FoL} and \cite{Naoi}.
\begin{prop}\label{globlocrel}  For $\lambda=\sum_{i=1}^n\lambda_i\omega_i\in P^+$, we have $${\rm{ch}}_{\gr}W(\lambda,0)=\frac{{\rm{ch}}_{\gr}W_{\loc}(\lambda,0)}{\prod_{i=1}^n(q;q)_{\lambda_i}},\qquad
\mathbb H(W(\lambda,0))= \frac{\mathbb H(W_{\loc}(\lambda,0))}{\prod_{i=1}^n(q;q)_{\lambda_i}}.$$\hfill\qedsymbol\end{prop}
\def \gs {{\lie sl}}
\subsection{} The final piece of information we need is on the graded character of the local Weyl modules. We restrict our attention to the case of $\lie{sl}_{r}$ but direct the interested reader towards \cite{Naoi} for the general case. There are a well--known family of $\bz_+$--graded  modules for the subalgebra $\lie n^+\otimes \bc[t]\oplus(\lie n^-\oplus\lie h)\otimes t\bc[t]$ of $\lie g[t]$ called the Demazure modules (see \cite{Ku} for an exposition). In the case of $\lie{sl}_r$ it was proved in \cite{S}  that the characters of certain Demazure  modules which are indexed by $P^+$ are  given by specializing the Macdonald polynomials at $t=0$. Moreover, these Demazure modules actually admit the structure of a $\lie g[t]$--module. The main result of \cite{CL} establishes that $W_{\loc}(\lambda)$ is graded isomorphic to  such  a Demazure module and can be summarized as follows.
\begin{thm}\label{chweyl} Assume that $\lie g$ is of type $\lie{sl}_{r}$ and let $\lambda=\sum_{i=1}^s\lambda_i\omega_i\in P^+$.
Then $$ \sum_{k\ge 0}[(W_{\loc}(\lambda,0)[k]: V(\mu)]q^k = \sum_{\ell\ge 0} \eta^\lambda_{\mu,\ell} (q),\qquad \ch_{\gr}W_{\loc}(\lambda,0)=\sum_{\mu\in P^+}\sum_{\ell\ge 0}\eta^\lambda_{\mu,\ell} (q)\ch_{\lie g}V(\mu),$$ where the $\eta^
\lambda_{\mu,ell}$ are as defined in Lemma \ref{crucial}.
\hfill\qedsymbol\end{thm}
Using Lemma \ref{crucial} we have  the following corollary.\begin{cor} \begin{gather*}\mathbb H(W_{\loc}(\lambda,0))\iffalse =\sum_{\mu\in P^+}\sum_{\ell\ge 0}\eta^{\lambda}_{\mu,\ell}(q)\dim V(\mu)=\sum_{\mu\in P^+}\sum_{\ell\ge 0}\eta^{\lambda}_{\mu,\ell}(q)s_\mu(1,\cdots,1)\\ \fi =P_\lambda(1,\cdots,1;q,0).\end{gather*}\end{cor}

\subsection{} We now have,
\begin{gather*}\sum_{k\ge 0}\sum_{\mu\in P^+}\mathbb H(W(\mu,0) ) [W_{\loc}(\mu,0): V(\lambda,k)]q^k
\\ =\sum_{\mu\in P^+}\left(\sum_{\ell\ge 0}\eta^\mu_{\lambda,\ell}(q)\right)\frac{\mathbb H(W_{\loc}(\mu,0))}{\prod_{i=1}^n(q;q)_{\mu_i}}\qquad
=\sum_{\mu\in P^+}\left(\sum_{\ell\ge 0}\eta^\mu_{\lambda,\ell}(q)\right)\frac{P_\mu(1,\cdots,1;q,0)
}{\prod_{i=1}^n(q;q)_{\mu_i}}\\ = \frac{s_\lambda(1,\ldots,1)}
{(q;q)^{r^2-1}_\infty}, \end{gather*} where the last equality is by using Lemma \ref{crucial}. Together with Lemma \ref{hproj}, we have now established that,
$$\mathbb H(P(\lambda,0))=\sum_{k\ge 0}\sum_{\mu\in P^+}\mathbb H(W(\mu,0) ) [W_{\loc}(\mu,0): V(\lambda,k)]u^k,$$ which is precisely the statement of Proposition \ref{red2}.

\subsection{} It remains to prove Theorem \ref{bggtr}. Set $M=P(\lambda_s,r)$ and let $M_\ell$, $\ell\in\bz_+$ be the canonical filtration of $M$. Then, it is clear that $$ M^k= P^k(\lambda_s,r)\cong M/M_{k+1},$$ which proves that the canonical filtration of $M^k$ is finite and in fact is given by the submodules $M_\ell/M_{k+1}$ where $0\le \ell\le k$. Moreover $$(M_\ell/M_{k+1})/(M_{\ell+1}M_{k+1})\cong M_\ell/M_{\ell+1},$$ and Theorem \ref{bggtr} is proved.

\iffalse

\subsection{}
Finally, we deduce the following corollaries of Theorem \ref{bgg}.
\begin{cor}
Suppose ${\lie g} = {\lie gl}_r$. Then we have
$$\ch_q M_\lambda = \sum_{\mu \in P^+} \ch_q  W^g_\mu \dim_q [W^l_\mu:
V_\lambda].$$

\begin{cor}
Let $S_\xi$ be the dual basis to Schur functions with respect to the Macdonald scalar product. Then
$$\sum_{n\ge 0} S_{\lambda,n}(y_1,\dots, y_r,0,\dots; 0,q) = \frac{\ch_{\gr}
P(\lambda,0)}{(y_1\dots y_r)_\infty}.$$
\end{cor}
\fi


\begin{thebibliography}{CPbanff}
\bibitem{AK} E.~Ardonne and R.~Kedem, {\emph Fusion products of Kirillov-Reshetikhin modules and fermionic multiplicity formulas}, J.Algebra, \textbf{308}, (2007), 270-294.
%\bibitem{BBCKL} M.~Bennett, A.~Berenstein, V.~Chari, A.~Khoroshkin and  S.~Loktev, \emph{BGG reciprocity for the current algebra of $\lie{sl}_{n+1}$}, in preparation.
\bibitem{BCM} M.~Bennett, V.~Chari, and N. ~Manning \emph{BGG reciprocity for  current algebras}, arXiv:1106.0347.
\bibitem{CFK} V.~Chari, G.~Fourier and T.~Khandai, \emph{A categorical approach to Weyl modules,} Transform. Groups \textbf{15} (2010), no. 3,  517--549.
\bibitem{CG} V.~Chari and J.~Greenstein, \emph{Current algebras, highest weight categories and quivers,} Adv. Math. \textbf{216} (2007), no. 2, 811--840.
    \bibitem{CL} V.~Chari and S.~Loktev, \emph{Weyl, Demazure and fusion modules for the current algebra of $\lie{sl}_{r+1}$.} Adv. Math. \textbf{207} (2006), 928--960.
    \bibitem{CPS} E.~Cline, B.~Parshall and L.~Scott, \emph{Finite dimensional algebras and highest weight categories,} J. Reine Angew. Math. \textbf{391} (1988), 85–-99.
\bibitem{CPweyl}V.~Chari and A.~Pressley, \emph{Weyl modules for classical and quantum affine algebras,} Represent. Theory \textbf{5} (2001), 191--223 (electronic).
    \bibitem{Donkin} S.~Donkin, \emph{Tilting modules for algebraic groups and finite dimensional algebras} A handbook of tilting theory, London Math. Society Lect. Notes 332 (2007), 215--257.
    \bibitem{FoL} G.~Fourier and P.~Littelmann, \emph{Weyl modules, Demazure modules, KR-modules, crystals, fusion products and limit constructions}, Adv. Math. \textbf{211} (2007), no. 2, 566--593.
 \bibitem{deFK}P.~Di Francesco, R.~Kedem, \emph{Proof of the combinatorial Kirillov-Reshetikhin conjecture }, . arXiv:0710.4415.
\bibitem{Garland} H.~Garland, \emph{The arithmetic theory of loop algebras,} J.~Algebra \textbf{53} (1978), 480--551.
        \bibitem{DR}  D.~Happel, C.M.~Ringel. {\em Tilted algebra} . Trans. Amer. Math. Soc. 274 (1982) 399-443.
        \bibitem{I} B.~Ion, {\em Nonsymmetric Macdonald polynomials and Demazure characters},
Duke Math. J {\bf 116}(2) 2003, 299--318
\bibitem{Ku} S.~Kumar, {\em Kac-Moody groups, their flag varieties, and representation theory}, Progress in Mathematics, \textbf{204} , (2002), Birkhauser, Boston.
        \bibitem{KodNaoi} R.~Kodera and K. Naoi, \emph{ Loewy series of Weyl modules and the Poincaré polynomials of quiver varieties}, arXiv:1103.4207.
\bibitem{Mac} I.~G.~McDonald, {\em Symmetric Functions and Hall Polynomials}, Oxford Univ. Press.
\bibitem{Naoi2} K.~Naoi,  \emph{  Weyl modules, Demazure modules and finite crystals for non-simply laced type}, arXiv:1012.5480.
\bibitem{Naoi} K.~Naoi, \emph{ Fusion products of Kirillov-Reshetikhin modules and the X = M conjecture}, arXiv:1109.2450.
    \bibitem{S} Y.~Sanderson, {\em On the Connection Between Macdonald Polynomials and Demazure Characters}, J. Algebraic
Comb. {\bf 11} (2000), 269--275.



\end{thebibliography}
\end{document}